\documentclass[10pt,reqno]{article}
\usepackage{bm}
\usepackage{amsthm}
\usepackage{amsmath}
\usepackage{amssymb}
\usepackage[round]{natbib}
\usepackage{color}
\usepackage{authblk}
\usepackage{graphicx}

\usepackage{geometry}
\newcommand{\rnc}{\renewcommand}
\newcommand{\nc}{\newcommand}
\newcommand{\mrm}{\mathrm}

\renewcommand{\hat}{\widehat}
\nc{\mb}{\mathbb}
\nc{\mc}{\mathcal}
\nc{\E}{\mb{E}}
\nc{\N}{\mb{N}}
\nc{\R}{\mb{R}}
\nc{\Q}{\mb{Q}}
\rnc{\P}{\mrm P}
\rnc{\d}{\mrm d}
\nc{\C}{\mc{C}}
\nc{\D}{\mc{D}}
\nc{\B}{\mc{B}}
\nc{\vbeta}{\bm \beta}
\nc{\vtheta}{\bm \theta}
\nc{\vX}{\bm X}
\nc{\vy}{\bm y}
\nc{\vU}{\bm U}
\nc{\vI}{\bm I}
\nc{\vE}{\bm E}
\nc{\ve}{\bm e}
\nc{\vV}{\bm V}
\nc{\vv}{\bm v}
\nc{\vS}{\bm S}
\nc{\vSigma}{\bm \Sigma}
\nc{\oPo}{\stackrel{\mrm p}{\rightarrow}}
\nc{\oWo}{\stackrel{w}{\rightarrow}}
\nc{\oDo}{\stackrel{d}{\longrightarrow}}
\nc{\eff}{\|F\|}
\nc{\red}{\color{red}}

\newtheorem{lemma}{Lemma}
\newtheorem{thm}{Theorem}
\newtheorem{cond}{Condition}

\newcommand{\intd}{\mathrm{d}}

\newcommand\blfootnote[1]{%
  \begingroup
  \renewcommand\thefootnote{}\footnote{#1}%
  \addtocounter{footnote}{-1}%
  \endgroup
} 

\begin{document}

\title{\Large \bf Wild Bootstrap based Confidence Bands\\ 
for Multiplicative Hazards Models}
\author[1$*$]{Dennis Dobler}
\author[2]{Markus Pauly}
\author[3]{Thomas H. Scheike} 

\affil[1]{Department of Mathematics, Vrije Universiteit Amsterdam, Netherlands.}
\affil[2]{Institute of Statistics, Ulm University, Germany.}
\affil[3]{Section of Biostatistics, University of Copenhagen, Denmark.}

\maketitle

\begin{abstract}
\noindent 
\blfootnote{${}^*$ e-mail:  d.dobler@vu.nl}
We propose new resampling-based approaches to construct asymptotically valid
	time simultaneous confidence bands for cumulative hazard functions in
	multi-state Cox models.  In particular, we exemplify the methodology in
	detail for the simple Cox model with time dependent covariates, where
	the data may be subject to independent right-censoring or
	left-truncation.  In extensive simulations we investigate their finite
sample behaviour. 
Finally, the methods are utilized to analyze an empirical example.  
\end{abstract}


\vfill
\vfill

 \newpage

%
\section{Introduction}
%

Wild bootstrap resampling has evolved as one of the state-of-the-art choices
for inferring cumulative incidences or hazards in nonparametric multi-state
models in event history analysis.  Starting with the initial papers by \cite{lin93} and \cite{lin1994confidence} for Cox models and \cite{lin97} for competing risks
set-ups, the basic idea is to consider martingale representations of the
nonparametric estimators (particularly, the Nelson-Aalen or Aalen-Johansen) and
to replace the non-observable martingale residuals $\d M_i(u)$ with randomly
weighted counting processes $G_i \d N_i(u)$.  This approach has been extended
in various directions, allowing for arbitrary multipliers 
\citep{beyersmann12b,dobler14, dobler2017non} and multiple, possibly recurrent, 
states
\citep{dobler2016nonparametric, bluhmki18a, bluhmki18b}.  
In the current paper
we like to transfer the latter results from the nonparametric case to 
semiparametric regression models.  
The most used regression model in survival analysis is Cox's proportional hazard model.
It is highly useful to estimate the survival function for
specific covariates, e.g., to show how the model predicts survival. 
The survival function of interest 
is then typically provided with point-wise confidence intervals which is implemented in all major software
packages. 
In reality, however, when interest is in the survival function as a whole, it would be preferable to report it together with uniform confidence intervals.
These so-called confidence bands describe the uncertainty of the whole 
survival function. 
This is often not done in practice because there are few programs that construct such 
uniform bands.
In addition, apart from only few exceptions such as \citep{lin97}, 
systematic evaluations of finite sample results, that demonstrate the performance of such bands, are rarely available in the literature.
We here provide such results and in addition investigate various new resampling bands that exhibit 
improved performance for smaller sample sizes compared to previously implemented bands for Cox's regression model.
This proportional hazards model \citep{cox72} is given by an individual-specific intensity function of 
the form 
\begin{align}
\label{eq:cox}
  \lambda_i(t, \vbeta_0) = Y_i(t) \lambda_0(t) \exp(\vX'_i(t) \vbeta_0).
\end{align}

Our main achievements are the introduction of valid resampling strategies that
jointly mimic the unknown distribution of baseline and parameter estimators for
Model~\eqref{eq:cox} and corresponding multi-state versions
\citep{martinussen06}. 
Different to existing approaches \citep{lin93,martinussen06},
we prove their theoretical validity by martingale-based
arguments which allow the simultaneous treatment of different mechanisms for
incomplete observations. In particular, the observations may be subject to
independent right-censoring and left-truncation.\\

\noindent{\bf How to resample?} There exist plenty of possible approaches 
to achieve the above tasks in Model~\eqref{eq:cox} with independent right-censoring alone. 
A first  corresponds to the nonparametric ansatz at the outset: here, we consider
martingale representations of the Breslow estimator for the cumulative hazard
and a parameter estimator that is found via a likelihood approach.
Then we replace the involved martingale residuals $\d M_i(u)$ with re-weighted counting processes
$G_i \d N_i(u)$ (e.g., \citealt{lin93}).  Since the latter do not take the
semiparametric nature into account, another possibility would be to replace them
with $G_i \d \hat M_i(u)$ (e.g., \citealt{spiekerman98}).  Here $\hat M_i$ are
estimators of the martingales $M_i$, that exploit the involved covariates and
allow for a greater range of applicability, for instance in rate estimations.

A novel and even more natural approach starts one step earlier by rewriting
 the score equations for the baseline function and the Euclidean parameter:
 after identifying a martingale representation of the score equations, both
 multiplier techniques from above lead to new equations which are solved by
 quantities depending on the $G_i$.  Hence, paralleling the same steps as for
 the original estimators, we receive their resampling counterparts in a
 primal way.

In all approaches we follow \cite{beyersmann12b} and allow for general wild bootstrap 
weights, i.e., the $G_i$ are i.i.d. random variables with zero mean and unit variance 
that are independent of the data. 

For ease of presentation we exemplify the new methodology mainly for the rather
simple Cox model but also explain their extensions to more general multi-state
or even other regression models.  The theoretical derivations for the wild bootstrap approaches thereby
utilize clever martingale arguments which are novel for bootstrapping in
semiparametric regression models. In particular, we prove that the wild
bootstrap counterparts share the martingale properties of the original
estimators -- and can therefore be handled in the same way, using convenient
martingale central limit theorems.  Thus, intricate derivations for verifying
conditional tightness are no longer required. Moreover, beneath theoretical
benefits, mirroring the martingale structure in the bootstrap world allows for
a simple interpretation and easy incorporation of missing mechanisms (such as
independent right-censoring or left-truncation). Consequently, our findings
allow for a wide range of applications, which to some extent will be discussed
in more detail in future papers.  Such martingale representations for the wild
bootstrap have first been made in the nonparametric context for resampling
Aalen-Johansen estimators; see \cite{dobler2016nonparametric} and
\cite{bluhmki18b} for details.\\


The paper is organized as follows: Section 2 outlines how estimation is done for Cox's regression model and
lists the technical conditions that are needed in proving the validity of the considered resampling approaches. 
Section 3 contains a description of the various wild bootstrap procedures that we consider here with theoretical 
statements about their validity. 
In addition, we discuss several important extensions to more general multi-state or other regression models. 
In Section 4 we present an extensive simulation study that compares the
various resampling procedures.
Section 5 has a brief demonstration of the methodology in a survival setting where 
interest is on constructing confidence bands for the survival function for patients with acute myocardial infarction.
Finally, we discuss the results in Section 6. 

All proofs are given in the Appendix, and these are a central part of this paper.
Their novelty lies in the fact that we are able to show the performance of our resampling methods using martingale 
methods which considerably simplify the technical arguments. 




%
\section{Joint large sample properties in the Cox model}
%

We consider the multiplicative Cox model \eqref{eq:cox} given by the intensity process 
$
  \lambda_i(t, {\vbeta}_0) = Y_i(t) \lambda_0(t) \exp(\vX'_i(t) \vbeta_0)
$ 
of the counting process $N_i(t)$ of subject $i$ given $\vX_i(t)$.  Here,
$Y_i(t)$ is the at-risk indicator of individual $i=1,\dots,n$ at time $t$,
$\lambda_0$ is the baseline hazard function, $\vX_i(t) = (X_{i1}(t), \dots,
X_{ip}(t))'$ is for each $t \geq 0$ a possibly time-dependent $p$-dimensional
vector of predictable covariates of individual $i$, and $\vbeta_0 = (\beta_1,
\dots, \beta_p)'$ is an unknown $p$-dimensional regression parameter
\citep{abgk93}.  Let $\tau > 0$ be a terminal evaluation time on the
treatment time-scale.  Throughout we assume that all $\vX_i(t)$ are contained
in a bounded set $F \subset \R^p$ and  denote the cumulative baseline hazard
function as $ \Lambda_0(t) = \int_0^t \lambda_0(s) \intd s$ which we assume to
be finite for all $t\leq \tau$.

A series of standard arguments typically leads to the {\it Breslow estimator} $\hat \Lambda_0$ for $\Lambda_0$ and the maximum likelihood parameter estimator $\hat \vbeta$ for $\vbeta_0$.
To illustrate this, let us simplify the derivations in \cite{scheike02} for the Cox-Aalen model to the present Cox model \eqref{eq:cox}: 
the score equation for the cumulative baseline function $\Lambda_0$ is given by
\begin{align}
\label{eq:score_baseline}
 \sum_{i=1}^n [ \d N_i(t) - Y_i(t) \exp(\vX_i'(t) \vbeta) \d \Lambda_0(t) ] = 0,
\end{align}
which is solved by 
$\hat \Lambda_0(t,\vbeta) = \sum_{i=1}^n \int_0^t J(u) S_0^{-1}(u,\vbeta) \d N_i(u)$.
Here, we used the definition of
$$S_k(t,\vbeta) = \sum_{i=1}^n \vX_i^{\otimes k}(t) Y_i(t) \exp(\vX_i'(t) \vbeta), \quad k=0,1,2, $$
where $\vy^{\otimes 2} = \vy\vy' \in \R^{p \times p}, \vy^{\otimes 1} = \vy \in \R^p$ and $\vy^{\otimes 0} = 1 \in \R$ for any vector $\vy \in \R^p$,
and $J(u)$ is the indicator that any individual is under risk shortly before $u$. 
For lucidity, the notion of $J(u)$ will be suppressed most of the time.
If we replace $\hat \Lambda_0(t,\vbeta)$ for $\Lambda_0$ in the score equation for $\vbeta$,
\begin{align}
\label{eq:score_beta}
  \sum_{i=1}^n \int_0^\tau \vX_i(u) Y_i(u) [ \d N_i(u) - \exp(\vX_i'(u) \vbeta) \d \Lambda_0(u)] = 0,
\end{align}
and define $\vE(t, \vbeta) = S_1(t,\vbeta) S_0^{-1}(t,\vbeta)$, 
we obtain a solvable score equation for $\vbeta$:
$$\vU_\tau(\vbeta) = \sum_{i=1}^n \int_0^\tau [ \vX_i(u) - \vE(u,\vbeta)] \d N_i(u) = 0.$$
Denoting its solution by $\hat \vbeta$, we also obtain the {\it Breslow estimator} $\hat \Lambda_0(t,\hat \vbeta)$ for the cumulative baseline hazard function. 
To explain their joint large sample properties 
we define by 
$$
\vI_\tau(\vbeta) = - D \vU_\tau(\vbeta) = \int_0^\tau \vV(u,\vbeta) \d N(u)
$$
the negative of the Jacobi-matrix of $\vU_\tau$, 
where $\vV(t, \vbeta) = S_2(t,\vbeta) S_0^{-1}(t,\vbeta) - \vE^{\otimes 2}(t,\vbeta)$
and $N = \sum_{i=1}^n N_i$. 
Recall that the covariates are assumed to be uniformly bounded. Therefore, 
it follows from Theorems~VII.2.2 and VII.2.3 of \cite{abgk93} that 
 $\sqrt{n} (\hat \vbeta - \vbeta_0)$ and  $\sqrt{n} (\hat \Lambda_0( \cdot, \hat \vbeta) - \Lambda_0(\cdot))$ are both asymptotically Gaussian as long as the following regularity conditions are fulfilled  which we assume throughout; see also Condition VII.2.1 \cite{abgk93}. Here and throughout, $\oPo$ denotes convergence in probability. 
\begin{cond}
 \label{cond:abgk}
 There exist a neighbourhood $\mc B$ of $\vbeta_0$ 
 and functions $s_0: [0,\tau] \times \mc B  \rightarrow \R$,
 $s_1: [0,\tau] \times \mc B  \rightarrow \R^p$,
 and $s_2: [0,\tau] \times \mc B  \rightarrow \R^{p \times p}$
 such that for each $k=0,1,2$:
 \begin{itemize}
  \item[(a)] $\sup_{\vbeta \in \mc B, t \in [0,\tau]} |n^{-1} S_k(t,\vbeta) - s_k(t,\vbeta) | \oPo 0;$
  \item[(b)] $s_k$ is a continuous function of $\vbeta \in \mc B$ uniformly in $t \in [0,\tau]$ and bounded on $[0,\tau] \times \mc B $;
  \item[(c)] $s_0(\cdot,\vbeta_0)$ is bounded away from zero on $[0,\tau]$;
  \item[(d)] $s_{k+1}(t,\vbeta) = \frac{\partial}{\partial \vbeta} s_{k}(t,\vbeta), k=0,1,$ for $\vbeta \in \mc B, t \in [0,\tau]$;
  \item[(e)] $\vSigma_\tau = \int_0^\tau \vv(t,\vbeta_0) s_0(t,\vbeta_0) \d \Lambda_0(t)$ is positive definite, where $\vv(t, \vbeta) = s_2(t, \vbeta) s_0^{-1}(t,\vbeta) - (s_1(t, \vbeta) s_0^{-1}(t,\vbeta))^{\otimes 2}$.
 \end{itemize}
\end{cond}
Note that (a) and (b) immediately imply convergence in probability for each $k=0,1,2$:
\begin{align}
 \label{eq:uniform_conv_sk_beta_hat}
 \sup_{t \in [0,\tau]} | S_k(t, \tilde \vbeta) - s_k(t,\vbeta_0)| \oPo 0
\end{align}
as long as $\tilde \vbeta \oPo \vbeta_0$. 
Carefully checking the proofs of the fore-mentioned theorems from \cite{abgk93}, we obtain asymptotic representations of the normalized estimators which will motivate the first bootstrap approaches in the following section:
\begin{align}\label{eq:as_expr_beta}
 \sqrt{n} (\hat \vbeta - \vbeta_0) & = \Big(\frac1n \vI_\tau(\vbeta_0)\Big)^{-1} \frac{1}{\sqrt{n}} \vU_\tau(\vbeta_0) + o_p(1)
 \\
 \sqrt{n} (\hat \Lambda_0( \cdot, \hat \vbeta) - \Lambda_0(\cdot)) 
  & = - \sqrt{n} (\hat \vbeta - \vbeta_0) \int_0^\cdot \ve(u,\vbeta_0) \d \Lambda_0(u) 
 + \sqrt{n} \int_0^\cdot S_0^{-1}(u,\vbeta_0) \d M(u) + o_p(1).\label{eq:as_expr_lambda}
\end{align}
Here, $\ve$ denotes the limit (in probability) of $\vE$, and $M(t) = \sum_{i=1}^n M_i(t) = \sum_{i=1}^n \Big( N_i(t) - \int_0^t \lambda_i(u, \vbeta_0) \d u\big)$ defines a square-integrable martingale in $t \in [0,\tau]$; cf. Section VII.2.2 in \cite{abgk93}.

\section{Wild bootstrap approaches and main theorems}

While inference about $\vbeta_0$ can be based on the asymptotic normality of its estimator (e.g., Martinussen and Scheike, 2006), the complicated limit process of the normalized Breslow estimator does not allow time simultaneous inference 
about the cumulative hazard function $\Lambda_0$ or functionals thereof (such as the survival function). To this end, we  propose two general approaches to establish asymptotically valid resampling strategies. Since $\hat \Lambda_0$ implicitly depends on $\hat \vbeta$, we have to ensure that their wild bootstrap counterparts mimic their distribution jointly.

\subsection{The `classical' wild bootstrap}
\label{sec:dirres}

The {\it first method} is inspired by the use of the wild bootstrap in \cite{beyersmann12b} and is in line with the resampling procedures of \cite{lin93} or \cite{spiekerman98} for the special choice of i.i.d. standard normal weights. 
This procedure is based on the above asymptotic representation of the normalized estimators and replaces the involved martingales $M_{i}(t) = N_i(t) - \int_0^t Y_i(u) \exp(\vX_i'(u) \vbeta_0) \d \Lambda_0(u)$ by $G_{i} N_{i}(t)$ or $G_{i}\hat M_{i}(t)$ together with plug-in estimators for all unknown quantities.
Here, $\d \hat M_i(u) = \d N_i(u) - Y_i(u) \exp(\vX_i'(u) \hat \vbeta) \d \hat \Lambda_0(u, \hat \vbeta)$
is an estimate of the martingale increment $\d M_i(u)$. 
We exemplify the idea for $G_i \d N_i(u)$: To this end, we introduce resampling versions of the score equation defining vector $\vU_\tau$ and the negative Jacobi matrix $\vI_\tau$:
\begin{align}
 \label{eq:wbs1-3}
  \vU_\tau^*(\hat \vbeta) = & \sum_{i=1}^n G_i \int_0^\tau (\vX_i(t) - \vE(t,\hat\vbeta)) \intd N_{i}(t) \text{ and} \\
  \label{eq:wbs1-4}
  \quad \frac1n \vI_\tau^*(\hat \vbeta) = & \frac1n \sum_{i=1}^n G_i^2 \int_0^\tau (\vX_i(t) - \vE(t,\hat \vbeta))(\vX_i(t)- \vE(t,\hat \vbeta))'  dN_{i}(t).
\end{align}
Following the above instruction we obtain from the asymptotic representations \eqref{eq:as_expr_beta}--\eqref{eq:as_expr_lambda} the following wild bootstrap counterparts of the normalized estimators:
\begin{align}
 \label{eq:wbs1-2}
 \sqrt{n} (\hat \vbeta^* - \hat \vbeta)  := & \Big(\frac1n \vI_\tau^*(\hat \vbeta)\Big)^{-1} \frac{1}{\sqrt{n}} \vU_\tau^*(\hat \vbeta), \\
 \label{eq:wbs1-2a}
 \sqrt{n} (\hat \Lambda_0^*(\cdot, \hat \vbeta^*) - \hat \Lambda_0(\cdot, \hat \vbeta)) = & - \sqrt{n} (\hat \vbeta^* - \hat \vbeta)' \int_0^\cdot \vE(u,\hat \vbeta) \hat \Lambda(\d u, \hat \vbeta) 
 \ \ + \ \ \sqrt{n} \sum_{i=1}^n G_i \int_0^\cdot S_0^{-1}(u, \hat \vbeta) \d N_i(u). 
\end{align}
Alternatively, the \cite{spiekerman98}-type martingale increment estimates $G_i \d \hat M_{i}(u)$ may replace $G_i \d N_i(u)$ 
in \eqref{eq:wbs1-3} and \eqref{eq:wbs1-2a}.
A bootstrap-type covariance estimate similar to \eqref{eq:wbs1-4} has been suggested by \cite{dobler14} in a nonparametric competing risks context. Here, it is additionally motivated from martingale arguments: 
defining $\vI_t^*$ and $\vU_t^*$ as in \eqref{eq:wbs1-3} and \eqref{eq:wbs1-4} with $\tau$ replaced with $t$, it turns out 
that $(I_t^*(\hat \vbeta))_{t \in [0,\tau]}$ is the optional variation process of the square-integrable {\it martingale} $(n^{-1/2} \vU_t^*(\hat \vbeta))_{t \in [0,\tau]}$;
see the appendix for details. To motivate a different resampling strategy, we finally note that both wild bootstrap procedures ignore the $o_p(1)$-terms in the asymptotic expansions (\ref{eq:as_expr_beta}) -- (\ref{eq:as_expr_lambda}).

\subsection{Wild bootstrapping the score equations}
\label{sec:esteq}

A {\it second}, possibly more natural wild bootstrap approach does not ignore the  $o_p(1)$ terms.
The idea is to replace martingale representations of score equations
with their multiplier counterparts. To this end, paralleling the approach of jointly solving two score equations to find the estimators for the parametric as well as the nonparametric model components, 
we first expand the score equation in \eqref{eq:score_baseline} to 
$\sum_{i=1}^n \d M_i(t) + \sum_{i=1}^n Y_i(t) [ \exp( \vX'_i(t)\vbeta_0 ) - \exp( \vX'_i(t)\vbeta ) ] \d \Lambda_0(t) = 0.$ 
A wild bootstrap counterpart thereof is now given by replacing $\d M_i(t)$ with $G_i \d N_i(t)$, $\vbeta_0$ with $\hat \vbeta$,
and $\Lambda_0(t)$ with $\hat \Lambda_0(t, \hat \vbeta)$:
\begin{align}
 & \sum_{i=1}^n G_i \d N_i(t) + \sum_{i=1}^n Y_i(t) [ \exp( \vX'_i(t) \hat \vbeta ) - \exp( \vX'_i(t)\vbeta ) ] \d \hat \Lambda_0(t, \hat \vbeta) \\
 & = \sum_{i=1}^n (G_i+1) \d N_i(t) - S_0(t,\vbeta) \d \hat \Lambda_0(t, \hat \vbeta) = 0.\nonumber
\end{align}
Now, keeping $\vbeta$ fixed, the ``solution'' for $\hat \Lambda_0(t, \hat \vbeta)$ is clearly
$$ \hat \Lambda_0^*(t,\vbeta) = \sum_{i=1}^n (G_i + 1) \int_0^t S_0^{-1}(t,\vbeta) \d N_i(t). $$
Next, to find an appropriate wild bootstrap version of $\hat \vbeta$,
we consider a martingale representation of the score equation~\eqref{eq:score_beta} for $\vbeta$:
\begin{align}
 \nonumber
 \sum_{i=1}^n \int_0^\tau \vX_i(t) Y_i(t) \d M_i(t) 
 + \sum_{i=1}^n \int_0^\tau \vX_i(t) Y_i(t) [\exp(\vX'_i(t) \vbeta_0) - \exp(\vX'_i(t) \vbeta)] \d \Lambda_0(t) = 0.
\end{align}
Again, a wild bootstrap version thereof is given by 
\begin{align*}
 & \sum_{i=1}^n G_i \int_0^\tau \vX_i(t) Y_i(t) \d N_i(t) 
 + \sum_{i=1}^n \int_0^\tau \vX_i(t) Y_i(t) [\exp(\vX'_i(t) \hat \vbeta) - \exp(\vX'_i(t) \vbeta)] \d \hat \Lambda_0(t,\hat \vbeta) \\
 & = \sum_{i=1}^n \int_0^\tau [ \vX_i(t) Y_i(t) G_i + \vE(t,\hat \vbeta)] \d N_i(t) 
 - \int_0^\tau S_1(t,\vbeta) \d \hat \Lambda_0(t,\hat \vbeta) = 0.
\end{align*}
Inserting $ \hat \Lambda_0^*(t,\vbeta)$ for $\hat \Lambda_0(t, \hat \vbeta)$ eventually yields the final {\it wild bootstrap score equation}
\begin{align}
 \vU_\tau^*(\vbeta) & = \sum_{i=1}^n \int_0^\tau [ \vX_i(t) Y_i(t) G_i + \vE(t,\hat \vbeta)] \d N_i(t) 
 - \int_0^\tau \vE(t,\vbeta) \sum_{i=1}^n (G_i+1) \d N_i(t) \nonumber \\
 & = \sum_{i=1}^n (G_i+1) \int_0^\tau [\vX_i(t) - \vE(t,\vbeta)] \d N_i(t) \stackrel{!}{=} 0. \label{eq:score_beta_wbs}
\end{align}
The last equality is due to $\sum_{i=1}^n \int_0^\tau [\vX_i(t) - \vE(t,\hat \vbeta)] \d N_i(t) = \vU_\tau(\hat \vbeta)=0$.
Define $\hat \vbeta^*$ as the solution of~\eqref{eq:score_beta_wbs} and note that $\vU_\tau^*(\hat \vbeta)$ coincides with formula \eqref{eq:wbs1-3}.
In almost the same way as in the proof of Theorem~VII.2.1 in \cite{abgk93} it can be shown
that the probability of the existence of $\hat \vbeta^*$ tends to one and that 
$\hat \vbeta^* - \hat \vbeta$ (conditionally) converge to zero in probability;
see also the proof of Theorem~\ref{thm:wbs} below for similar arguments.

Finally, a wild bootstrap version of the Breslow estimator is obtained via $\hat \Lambda_0^*(\cdot, \hat \vbeta^*)$ with  normalized version
\begin{align}
 \label{eq:breslow_wbs}
 \sqrt{n} (\hat \Lambda_0^*(t, \hat \vbeta^*) - \hat \Lambda_0(t, \hat \vbeta)) 
 & = \sqrt{n} \sum_{i=1}^n (G_i+1) \int_0^t [S_0^{-1}(u,\hat \vbeta^*) - S_0^{-1}(u, \hat \vbeta)] \d N_i(u)  
  +  \sqrt{n} \sum_{i=1}^n G_i \int_0^t S_0^{-1}(u,\hat \vbeta) \d N_i(u)
 \\
 & = \sqrt{n} \sum_{i=1}^n \int_0^t [S_0^{-1}(u,\hat \vbeta^*) - S_0^{-1}(u, \hat \vbeta)] \d N_i(u)
 + \sqrt{n} \sum_{i=1}^n G_i \int_0^t S_0^{-1}(u,\hat \vbeta^*) \d N_i(u).
 \nonumber
\end{align}
A Taylor expansion around $\hat \vbeta$ of the first term on the far right-hand side
and the martingale property of the second term reveal the striking similarity to decomposition~\eqref{eq:wbs1-2a}.
However, the current wild bootstrap approach does not ignore the $o_p(1)$ term resulting from the Taylor expansion. 
Another nice property of this ``estimating equation'' approach is the similar treatment for bootstrap and original estimator which is in line with general recommendations for constructing resampling algorithms \citep{beran1991asymptotic, efron1994introduction}. 
As above we have by the mean value theorem (cf. \citealt{feng13})
\begin{align*}
 - \vU_\tau^*(\hat \vbeta) = \vU_\tau^*(\hat\vbeta^*) - \vU_\tau^*(\hat \vbeta) = \tilde D \vU_\tau^*(\tilde \vbeta) (\hat \vbeta^* - \hat \vbeta),
\end{align*}
where $\tilde D \vU_\tau^*(\tilde \vbeta) = (\nabla (\vU_\tau^{*})^{(1)} (\vbeta_1)', \dots, \nabla (\vU_\tau^{*})^{(p)} (\vbeta_p)')' $
and each $\vbeta_i$ is on the line segment between $\hat\vbeta^*$ and $\hat\vbeta$.

%
%
%
%

%
\subsection{Consistency and confidence bands for the cumulative hazard}
\label{sec:cbs}
%
To prove (asymptotic) validity of both resampling strategies (based on asymptotic expansions or score equations) the following result is needed.
\begin{lemma}
\label{lemma:prelim}
Under Conditions~\ref{cond:abgk}(a)-(e) it holds that, given all observations,
\begin{enumerate}
 \item $n^{-1/2} \vU_\tau^*(\hat \vbeta)$ is asymptotically normally distributed,
 \item $- \frac1n \tilde D \vU_\tau^*(\tilde \vbeta) \oPo \Sigma$ whenever $\tilde \vbeta \oPo \vbeta_0$,
 \item $\| \hat \vbeta^* - \hat \vbeta \| \oPo 0$
\end{enumerate}
 as $n \rightarrow \infty$
 in probability if the resampling is done via method~\eqref{eq:wbs1-2} or \eqref{eq:score_beta_wbs}. 
\end{lemma}

The next theorem constitutes that both resampling approaches utilizing the \cite{lin93} approach (i.e. with $G_i \d N_i$), have the correct asymptotic behaviour. Therein, $d$ denotes a distance that metrizes weak convergence on $\R^p \times D[0,\tau]$, e.g. the Prohorov distance \citep{dudley2002real}, and 
$\mathcal{L}(T)$ and $\mathcal{L}(T|\text{data})$ are the unconditional and conditional distribution of a random variable $T$, respectively.

\begin{thm}
\label{thm:wbs}
 Under Condition~\ref{cond:abgk} it holds for both resampling strategies \eqref{eq:wbs1-2} or \eqref{eq:score_beta_wbs} that 
 the asymptotic distributions of  $\sqrt{n} (\hat\vbeta^* - \hat \vbeta, \hat\Lambda^*_0 - \hat \Lambda_0)$ and 
 $\sqrt{n} (\hat\vbeta - \vbeta, \hat\Lambda_0 - \Lambda_0)$ coincide, i.e. 
\begin{equation}
d \left( \mathcal{L}\Big(\sqrt{n}\big(\hat\vbeta^* - \hat \vbeta, \hat\Lambda^*_0 - \hat \Lambda_0\big)| \text{data}\Big) ,
 \mathcal{L}\big(\sqrt{n}\big(\hat\vbeta - \vbeta, \hat\Lambda_0 - \Lambda_0\big)\big)\right) \oPo 0
\end{equation}
as $n \rightarrow \infty$, where $\hat\Lambda_0(t) = \hat\Lambda_0(t, \hat \vbeta)$ and $\hat\Lambda^{*}_0(t) = \hat{\Lambda}^{*}_0(t, \hat{\vbeta}^{*})$.
\end{thm}
The asymptotic variance function $t \mapsto \sigma^2(t)$ of $\sqrt{n}(\hat\Lambda_0 - \Lambda_0)$ (and thus also of 
$\sqrt{n}(\hat\Lambda^*_0 - \hat \Lambda_0)$) can be found in \citet[Corollary~VII.2.4]{abgk93}, where also a consistent estimator $\hat \sigma^2(t)$ is given. 
In our simulation study in Section~\ref{sec:simus}, our choice of a wild bootstrap counterpart of $\hat \sigma^2(t)$ was the empirical variance function of the obtained wild bootstrap realizations of $\sqrt{n}(\hat\Lambda^*_0 - \hat \Lambda_0)$. 
We also studied variance estimators based on direct resampling of $\hat{\sigma}^2$  involving squared multipliers (results not shown) as proposed in \cite{dobler14}. However, the empirical versions performed preferably.

The theorem is proven in the Appendix. Here, we use it to construct time-simultaneous confidence bands for $\Lambda_0$ on fixed intervals $I=[t_1,t_2]\subset[0,\tau]$. In particular, we obtain results similar  to those of \cite{lin1994confidence}: denoting by $\phi$ a continuously differentiable function 
we get confidence bands of asymptotic level $1-\alpha$ for $\Lambda_0$ on $I$ as 
$$
\phi^{-1} [\phi(\hat \Lambda_0(t,\hat \vbeta)) \mp c^*_\phi(\alpha)/g_n(t)],
$$
where $g_n: I \rightarrow (0,\infty)$ is a possibly random weight function. Typical choices are 
$$
g_n^{(1)}(t) =  \sqrt{n}/\widehat{\sigma}(t) \quad \text{and}\quad g_n^{(2)}(t) =  \sqrt{n}/(1+\widehat{\sigma}^2(t))
$$
in case of the transformation $\phi_1(x) = x$, and
$$
\tilde g_n^{(1)}(t) =  \sqrt{n}\hat\Lambda_0(t,\hat \vbeta)/\widehat{\sigma}(t) \quad \text{and}\quad \tilde g_n^{(2)}(t) =  \sqrt{n}\hat\Lambda_0(t,\hat \vbeta)/(1+\widehat{\sigma}^2(t))
$$
for the transformation $\phi_2(x) = \log(x)$. The resulting confidence bands correspond to the so-called {\it equal precision} (for $g_n^{(1)}$ or $\tilde g_n^{(1)}$) and {\it Hall-Wellner bands} (for $g_n^{(2)}$ or $\tilde g_n^{(2)}$), respectively. 
Finally, the value of $c^*_\phi(\alpha) = c^*_{\phi_1}(\alpha)$ has been chosen as the
$(1-\alpha)$ quantile of the conditional distribution of
$\sup_{t\in I}  g^*(t) |\hat \Lambda^*(t,\hat \vbeta) - \hat \Lambda(t,\hat \vbeta)|$,
and the na\"ive choice for $c^*_{\phi_2}(\alpha)$ would have been the corresponding quantile of
$\sup_{t\in I} \tilde g^*(t) |\log(\hat \Lambda^*(t)) - \log(\hat \Lambda(t))|$.
Here, $g^*(t)$ and $\tilde g^*(t)$ are the wild bootstrap analogues of $g_n^{(j)}(t)$ and $\tilde g_n^{(j)}(t)$, respectively, $j \in \{1,2\}$.
However, this choice of $c^*_{\phi_2}(\alpha)$ results in some numerical instabilities, 
which is why we preferred the asymptotically equivalent choice 
$c^*_{\phi_2}(\alpha) = c^*_{\phi_1}(\alpha)$.

Here, the ``wild bootstrap analogues'' refer to the use of $\hat{\Lambda}^*_0$ for any of the bootstrap strategies~\eqref{eq:wbs1-2} or \eqref{eq:score_beta_wbs}, and its corresponding empirical variances.
It follows from Theorem~\ref{thm:wbs} that all confidence bands are valid for large sample sizes. 
To additionally asses their small sample properties, we compare them in  Monte-Carlo simulations in Section~\ref{sec:simus}. 
There, we also analyze the analogue behaviour of the resampling approaches based on $\d \hat M_i$.

%
\subsection{Extensions to more general models and more on inference}
%

After having carefully checked the arguments used to establish the wild bootstrap consistency for the Cox survival model~\eqref{eq:cox}, it is apparent that the same approach directly carries over to more general models in multi-state set-ups. 
In particular, as long as the counting process martingale methods can be mimicked with the help of wild bootstrap multipliers,  
 the asymptotics of the resampled estimators can be argued in almost the same way as for the original estimators. 
Thus, the above methodology can straightforwardly be extended to multi-state models with $K$ states and multiplicative intensity processes 
\begin{equation}
\label{eq:multi}
 \lambda_{ih}(t, \vtheta) = Y_{ih}(t) \lambda_{h0}(t,\gamma) \exp(\vX'_{i}(t) \vbeta_0)
\end{equation}
for each transition $h=1,\dots,K(K-1)$, where $\vtheta = (\gamma,\vbeta_0')'$. 
Different to above this model allows for an arbitrary number of transitions between different states. However, following \cite{dobler2016nonparametric} and \cite{bluhmki18b}, the above wild bootstrap approach can also be applied here. 
The only major change is to replace the currently used multipliers $G_i$ by more general white noise processes $(G_{ih}(u))_u$ with zero mean and unit variance \citep{bluhmki18b} to randomly weight the increments of the counting processes, leading to $G_{ih}(u) \d N_{ih}(u) $. Since the martingale concept is still working in this case it can again be shown 
that the wild bootstrap mimics the joint limit distribution of the parameter and 
multivariate hazard transition estimators. 
Indeed, \cite{dobler2016nonparametric} and \cite{bluhmki18b} have shown that, for different transitions $h$ and $h'$ and thus independent white noise processes $(G_{ih}(u))_u$ and $(G_{ih'}(u))_u$, 
the processes $\int_0^t f_{ih}(u) G_{ih}(u) \d N_{ih}(u)$ and $\int_0^t f_{ih'}(u) G_{ih'}(u) \d N_{ih'}(u)$
define orthogonal square-integrable martingales in $t$ with respect to the filtration
$$ (\mc F_t := \sigma\{ Y_{ih}(u), N_{ih}(u), Y_{ih'}(u), N_{ih'}(u): 0 \leq u \leq \tau; \ G_{ih}(v), G_{ih'}(v): 0 \leq v \leq t \} )_t. $$
Here, $f_{ih}$ and $f_{ih'}$ are predictable random functions with respect to this filtration, i.e. in particular, they may be data-dependent.
The predictable variation processes of the above martingales are
$\int_0^t f_{ih}^2(u) \d N_{ih}(u)$ and $\int_0^t f_{ih}^2(u) \d N_{ih'}(u)$.
This property nicely reflects the situation for the original estimators, 
as the corresponding counting process martingales
$\int_0^t f_{ih}(u) \d M_{ih}(u)$ and $\int_0^t f_{ih'}(u) \d M_{ih'}(u)$
are orthogonal and square-integrable as well with predictable variation processes
$$\int_0^t f_{ih}^2(u) Y_{ih}(u) \lambda_{h0}(u, \gamma) \exp(\vX'_i(u) \vbeta_0)\d u \quad \text{and}\quad \int_0^t f_{ih}^2(u) \lambda_{h'0}(u, \gamma) \exp(\vX'_i(u) \vbeta_0)\d u. $$
In this sense, not only the wild bootstrap martingales resemble the original counting process martingales 
well but also the predictable variation processes of the wild bootstrap martingales 
are estimates of the original predictable variation processes.
Using these findings in combination with the arguments presented in the proofs in the appendix,
it is apparent that also in such more general multi-state set-ups the arguments for the large-sample properties of the estimators easily transfer to their wild bootstrap versions, as long as the original estimators allow for martingale representations.
Therefore, these arguments even extend to more general models such as the Cox-Aalen multiplicative-additive intensity model \citep{scheike02} or the \cite{fine99} model for subdistribution functions.

Also, the incorporation of certain filtered (e.g., right-censored) observations is again allowed and this yields several important inferential applications: apart from confidence bands for cumulative transition hazards or incidence functions (which are functionals thereof), tests for null hypotheses formulated in terms the parameters can be constructed as well. 
Here, new  bootstrap-based versions of score or Wald-type test statistics \citep{martinussen06} may be employed to ensure a proper finite sample behaviour. However, 
a detailed evaluation of all these applications would need additional extensive simulations and further elaborations. As a matter of lucidity, we leave them to 
future research and we focus below on the simple Cox model \eqref{eq:cox} to exploit the impact of the proposed methods in  simulations.

\section{Simulation study}
\label{sec:simus}
%

To compare the performances of the various resampling approaches described in
Section~\ref{sec:cbs}, we conducted a simulation study in which we covered
situations of small to large sample sizes: $n= 100, 200, 400$.  The generated
data follow the Cox survival model with baseline hazard rate $\lambda_0 \equiv
1$, one-dimensional covariates $X_i \stackrel{\text{i.i.d.}}{\sim}
\mc{N}(0,16)$ which are normally distributed with standard deviation 4, and
regression parameter $\beta=0.3$.  The censoring times are the minima of $\tau = 3$ and
standard exponentially distributed random variables.  The considered time
interval, along which 95\% confidence bands for the cumulative baseline hazard
function shall be constructed, was $[t_1,t_2] = [0.5,3]$.  Here we chose the
start time of $t_1=0.5$ because ``the approximations tend to be poor for
$t$ close to 0'' \citep[p. 77]{lin1994confidence}.  As wild bootstrap
multipliers $G_1, \dots, G_n$, we considered the common choice $G_i
\stackrel{\text{i.i.d.}}{\sim} \mc{N}(0,1)$, as well as centered unit Poisson
variables $G_i \stackrel{\text{i.i.d.}}{\sim} Poi(1) - 1$ with unit skewness,
and centered unit exponential variables $G_i \stackrel{\text{i.i.d.}}{\sim}
Exp(1) - 1$ which have a skewness of 2.  We simulated all the confidence bands
for the cumulative baseline hazard function that were introduced in
Section~\ref{sec:cbs}, i.e.\  $\log$- and non transformed Hall-Wellner and equal
precision bands.  In particular, we also considered both resampling approaches
in which the martingale increments $\d M_i$ were replaced with $G_i \d N_i$ or
$G_i \d \widehat M_i$, denoted in Tables~\ref{tab:norm}--\ref{tab:pois} as
``$\d N$'' and ``$\d M$'', respectively, and also both kinds of resampling algorithms,
the direct resampling method of Section~\ref{sec:dirres} and the method of
Section~\ref{sec:esteq} in which the estimating equations were bootstrapped.
All of these bands were compared with the confidence band for $\Lambda_0$ that
one obtains from the \texttt{cox.aalen} function in the \texttt{R} package
\emph{timereg}.  For each considered set-up and type of band, we constructed
10,000 confidence bands, each of which was based on 999 wild bootstrap
iterations.
\begin{table}[ht]
\centering
\begin{tabular}{|cc|c|c|c|c|c|c|c|c|c|}
  \hline
  &  &  & \multicolumn{4}{|c|}{Hall-Wellner} & \multicolumn{4}{|c|}{equal precision}  \\ 
   &  & timereg & \multicolumn{2}{|c|}{estimating} & \multicolumn{2}{|c|}{direct} & \multicolumn{2}{|c|}{estimating} & \multicolumn{2}{|c|}{direct} \\ 
   & resampling & standard & \multicolumn{2}{|c|}{equation} & \multicolumn{2}{|c|}{resampling} & \multicolumn{2}{|c|}{equation} & \multicolumn{2}{|c|}{resampling} \\ 
  $n$ & approach & band & id & log & id & log & id & log & id & log \\ \hline
  100 & $\d N$ & 89.5 & 88.6 & 95.5 & 88.8 & 95.6 & 89.5 & 96.7 & 88.8 & 96.3 \\ 
      & $\d M$ & 86.5 & 88.7 & 95.4 & 88.7 & 95.5 & 89.4 & 96.6 & 88.8 & 96.1 \\ 
  200 & $\d N$ & 93.2 & 92.3 & 95.7 & 92.3 & 95.8 & 92.6 & 96.3 & 92.3 & 96.1 \\ 
      & $\d M$ & 91.7 & 92.2 & 95.7 & 92.3 & 95.7 & 92.4 & 96.4 & 92.2 & 95.9 \\ 
  400 & $\d N$ & 95.8 & 94.5 & 96.2 & 94.6 & 96.1 & 94.7 & 96.3 & 94.5 & 96.2 \\ 
      & $\d M$ & 95.1 & 94.4 & 96.2 & 94.5 & 96.2 & 94.6 & 96.6 & 94.6 & 96.4 \\ 
   \hline
\end{tabular}
\caption{Simulated coverage probabilities (in \%) of various 95\% confidence bands for the baseline cumulative hazard function and sample sizes $n = 100, 200, 400$ with standard normal wild bootstrap multipliers and empirical variance estimators.} 
\label{tab:norm}

\begin{tabular}{|cc|c|c|c|c|c|c|c|c|c|}
  \hline
 &  &  & \multicolumn{4}{|c|}{Hall-Wellner} & \multicolumn{4}{|c|}{equal precision}  \\ 
   &  & timereg & \multicolumn{2}{|c|}{estimating} & \multicolumn{2}{|c|}{direct} & \multicolumn{2}{|c|}{estimating} & \multicolumn{2}{|c|}{direct} \\ 
   & resampling & standard & \multicolumn{2}{|c|}{equation} & \multicolumn{2}{|c|}{resampling} & \multicolumn{2}{|c|}{equation} & \multicolumn{2}{|c|}{resampling} \\ 
  $n$ & approach & band & id & log & id & log & id & log & id & log \\ \hline
  100 & $\d N$ & 89.5 & 90.0 & 96.6 & 90.0 & 96.5 & 94.3 & 99.2 & 93.4 & 98.9 \\ 
      & $\d M$ & 86.5 & 90.2 & 96.7 & 89.8 & 96.4 & 94.3 & 99.2 & 92.3 & 98.2 \\ 
  200 & $\d N$ & 93.2 & 92.9 & 96.3 & 92.9 & 96.2 & 95.1 & 98.4 & 94.6 & 98.0 \\ 
      & $\d M$ & 91.9 & 93.4 & 96.6 & 93.3 & 96.4 & 95.1 & 98.6 & 94.2 & 97.9 \\ 
  400 & $\d N$ & 95.8 & 94.8 & 96.3 & 94.8 & 96.3 & 96.0 & 97.6 & 95.7 & 97.3 \\ 
      & $\d M$ & 95.1 & 94.8 & 96.3 & 94.7 & 96.4 & 96.0 & 97.7 & 95.5 & 97.3 \\ 
   \hline
\end{tabular}
\caption{Simulated coverage probabilities (in \%) of various 95\% confidence bands for the baseline cumulative hazard function and sample sizes $n = 100, 200, 400$ with standard exponential wild bootstrap multipliers and empirical variance estimators.} 
\label{tab:exp}

\begin{tabular}{|cc|c|c|c|c|c|c|c|c|c|}
  \hline
  &  &  & \multicolumn{4}{|c|}{Hall-Wellner} & \multicolumn{4}{|c|}{equal precision}  \\ 
   &  & timereg & \multicolumn{2}{|c|}{estimating} & \multicolumn{2}{|c|}{direct} & \multicolumn{2}{|c|}{estimating} & \multicolumn{2}{|c|}{direct} \\ 
   & resampling & standard  & \multicolumn{2}{|c|}{equation} & \multicolumn{2}{|c|}{resampling} & \multicolumn{2}{|c|}{equation} & \multicolumn{2}{|c|}{resampling} \\ 
  $n$ & approach & band & id & log & id & log & id & log & id & log \\ \hline
  100 & $\d N$ & 89.5 & 88.9 & 95.7 & 88.8 & 95.6 & 91.0 & 97.5 & 90.0 & 97.0 \\ 
      & $\d M$ & 86.5 & 88.8 & 95.6 & 89.1 & 95.6 & 91.0 & 97.6 & 89.6 & 96.9 \\ 
  200 & $\d N$ & 93.2 & 92.3 & 95.8 & 92.5 & 95.7 & 93.4 & 97.0 & 92.8 & 96.6 \\ 
      & $\d M$ & 91.9 & 92.9 & 96.1 & 92.9 & 96.2 & 93.4 & 97.2 & 92.9 & 96.7 \\ 
  400 & $\d N$ & 95.8 & 94.6 & 96.2 & 94.6 & 96.2 & 95.0 & 96.7 & 94.7 & 96.5 \\ 
      & $\d M$ & 95.1 & 94.5 & 96.2 & 94.5 & 96.2 & 95.1 & 96.8 & 94.8 & 96.6 \\ 
   \hline
\end{tabular}
\caption{Simulated coverage probabilities (in \%) of various 95\% confidence bands for the baseline cumulative hazard function and sample sizes $n = 100, 200, 400$ with standard Poisson wild bootstrap multipliers and empirical variance estimators.} 
\label{tab:pois}

\end{table}
The obtained empirical coverage probabilities are given in Tables~\ref{tab:norm}--\ref{tab:pois}.

%
%
%

We note that when the sample size is $400$ all methods gives a reasonable performance. When the sample size is smaller there are notable
differences, and it seems that the $\log$-transform does improve the performance in this case.
Whether the bootstrap is based on $\d N$ or $\d M$ does not seem important, and in terms of computations it is considerably easier and faster
to use the multipliers based on $\d N$. 

Even though there are strong theoretical and practical advantages of the Poisson variables over standard normal multipliers in nonparametric competing risks models \citep{dobler2017non},
the choice of the bootstrap multipliers does not seem highly important here.
Also, the choice of the particular resampling method, be it the direct approach of Section~\ref{sec:dirres} or the estimating equation approach of Section~\ref{sec:esteq}, does not seem to have a clearly positive or negative impact on the outcomes.

Finally, we would like to note that our simulation results are only partially comparable to those of \cite{lin1994confidence}:
here, we construct confidence bands for the baseline cumulative hazard function, i.e. for an individual with covariate $X_i=0$,
whereas they consider bands for survival curves for multiple covariates and their utilized transformations result in different bands.
Overall, however, the empirical reliability of the bands in both simulation studies, i.e. theirs and ours, are approximately the same.

%
\section{Data example}
%


In this section we briefly demonstrate how the confidence bands should
be used in a standard survival setting. The key point is that they
are most often the ones of interest unless focus is on a particular 
survival probability at a specific time such as for example $5$ year survival. 

We consider the TRACE study \citep{trace} 
where interest is on survival after acute myocardial infarction for 1878
consecutive patients included in the study. The data-set is available in 
the \emph{timereg} \texttt{R}-package. Here for sake of illustration we focus interest on
the covariates diabetes (1/0), sex and age.  
Due to the large sample size, we decided to use the \emph{timereg}-bands and show the 
survival predictions with uniform 95\% equal precision bands  based on the identity
transformation (broken lines) and standard normal multipliers. We also computed 95\% point-wise confidence intervals (dotted lines).
We depict the confidence bands for a male with average age
(66.9 years) and with or without diabetes, as well as the standard 95\% point-wise 
confidence intervals.  
\begin{figure}[ht]
  \centering
  \includegraphics[width=0.7\textwidth]{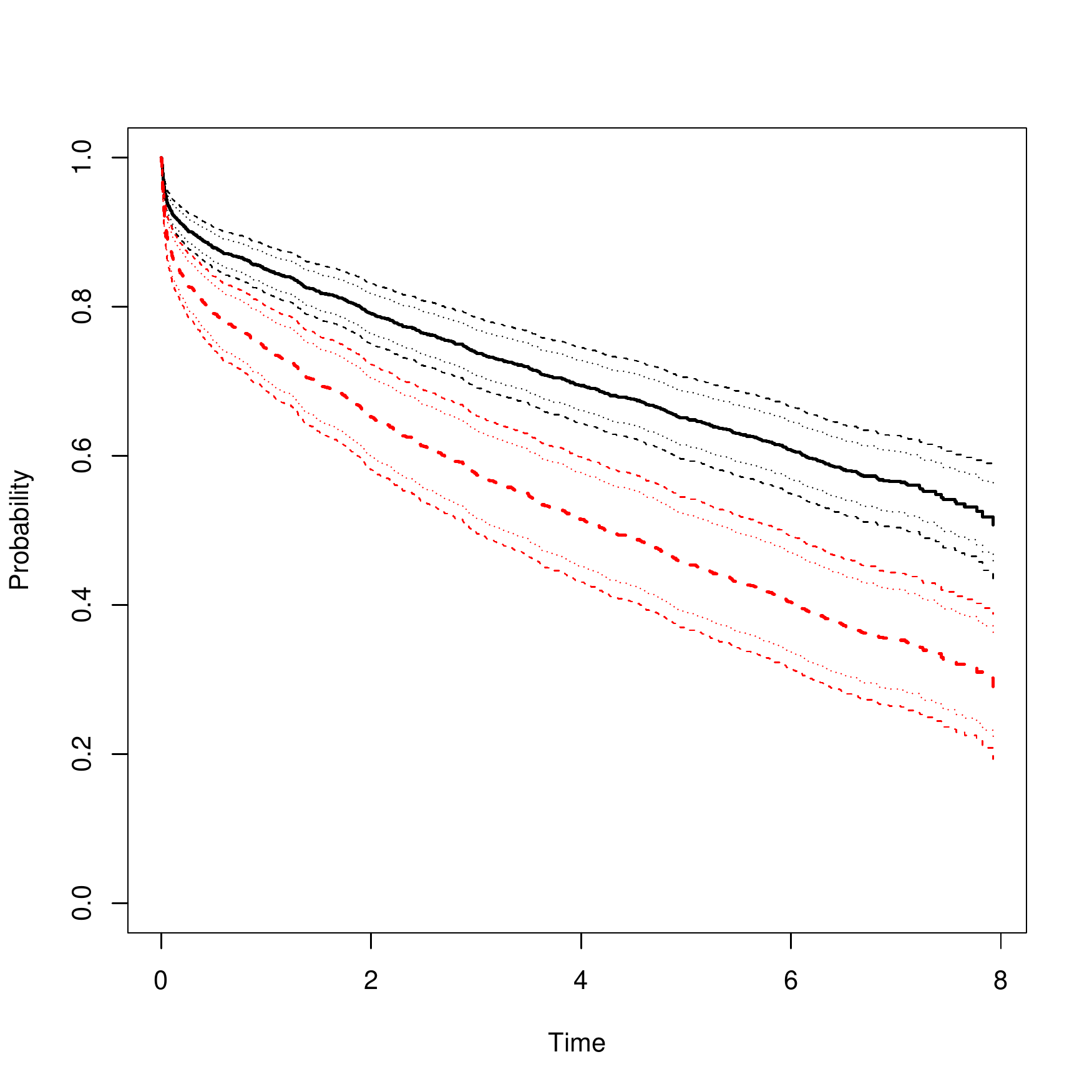}
	\caption{Survival function estimates for males with average age, with (fat broken line) or without diabetes (fat solid line), 
	and 95 \% confidence intervals (dotted lines) as well as
	95 \% bootstrap confidence bands (broken lines).}
  \label{f:trace-survival}
\end{figure}
We note that the hazard ratio related to diabetes is $1.82$ with 95\% confidence
interval $(1.50,2.18)$. Thus reflecting that diabetes is a factor that leads to
increased mortality. More interestingly, seen in connection with absolute level of 
mortality, this is then reflected in our estimated survival curves
for males with average age and with diabetes (lower broken fat curve, with confidence bands and intervals) 
or without diabetes (upper solid fat curve). We note that the bands are a bit wider than the point-wise intervals. As the latter do not provide simultaneous coverage 
the bands should be used to provide uncertainty about the entire survival curve as 
shown in Figure 1.

We finally illustrate how
the joint asymptotic distribution of the baseline and the covariates
can be used with other functionals. To this end, consider the restricted residual mean 
$$
\Psi(\Lambda_0)=\int_0^\tau \exp(-\Lambda_0(s))\d s
$$ 
with estimator $\Psi(\hat \Lambda_0)$. To get a description of its uncertainty based on the wild bootstrap constructions
we can simply apply the functional to the obtained bootstrap samples. 
It follows that 
$\sqrt{n}(\Psi(\hat \Lambda_0) - \Psi(\Lambda_0))$ has the same asymptotic 
distribution as 
$\sqrt{n}(\Psi(\hat \Lambda_0^*) - \Psi(\hat \Lambda_0))$ due to Hadamard differentiability of
the functional. Thus, we can easily construct symmetric 95\% confidence intervals for 
the restricted residual mean and their differences based on the bootstrap. The key point being that these are 
very easy to get at when the bootstrap estimates are at hand. 

For example, using the direct wild bootstrap approach based on $dN$ and standard normal multipliers,
we find that males with diabetes have a restricted residual mean within the first $5$ years at $3.87 (3.74,4.00)$ for males without diabetes and 
$3.15 (2.91,3.41)$ with diabetes. Males with diabetes thus lose $0.71 (0.49,0.93)$ years within the first $5$ years.
In a similar way confidence intervals for other functionals can be obtained by means of the continuous mapping theorem or the functional delta method.


%
\section{Discussion and further research}
%


Despite their importance, confidence bands are not used much in practice even 
though there is considerable interest in making survival predictions based on
semiparametric regression models such as the Cox model. This is probably due 
to the fact that the key software solutions do not have confidence bands 
implemented in this setting
. The aim of this work is to investigate some natural and simple wild 
bootstrap approaches for filling this gap. In particular, we have shown in the Appendix
that the proposed bootstrap solutions do asymptotically have the desired properties. 
A key point in our proofs is the fact that we show the properties of our 
bootstrap procedures relying solely on martingale arguments. 
This enormously facilitates the transfer of the classical proofs for the estimators to their wild bootstrap counterparts.
It became apparent that this approach generalizes to much more complex models as long as they admit a martingale structure for the involved counting processes.
This covers for example Cox models in multi-state models or Fine-Gray regression models for subdistribution functions.
A future work will focus on how the procedure can be adapted to more complex designs.

In addition, we consider the finite sample performance of various confidence bands and we see that, 
when the sample size is too small, one needs to be cautious when constructing such bands. 
When the sample size is reasonable, however, the bands perform well and should be the
preferred way of illustrating the uncertainty of the survival curves. 

Another, nice feature of the bootstrap approach is that it provides a very simple 
tool for constructing confidence intervals for functionals of the parameters of 
interest. We illustrated this by computing the restricted residual mean based 
on estimates from the Cox model.

\section*{Acknowledgements}
Markus Pauly likes to thank for the support from the German Research Foundation (Deutsche Forschungsgemeinschaft).

\appendix

\section*{Appendix}

%
\section{Proofs}
%

\begin{proof}[Proof of Lemma~\ref{lemma:prelim}]
 To a large extent, it is possible to parallel the martingale arguments as used in the proofs of Theorems~VII.2.1 and~VII.2.2 in \cite{abgk93}.
 We show the proof for the resampling scheme~\eqref{eq:score_beta_wbs} only; 
 once it has been understood how martingale methods can be applied here,
 it will be apparent how to conduct the proof for the classical wild bootstrap scheme~\eqref{eq:wbs1-2} which entirely consists of martingales.

 \medskip

\noindent \textbf{Proof of 3.}
 We introduce the process $C_t^*(\vbeta) = \sum_{i=1}^n (G_i + 1) [\int_0^t \vbeta' \vX_i(s) \d N_i(s) - \int_0^t \log {S_0(s, \vbeta)} \d N_i(s)]$ such that $\nabla C_\tau^*(\vbeta) = \vU_\tau^*(\vbeta)$, where $\nabla$ again denotes the gradient with respect to $\vbeta$.
 We wish to analyse the asymptotic behaviour of the process
 \begin{align*}
  X^*(t, \vbeta) = \frac1n (C_t^*(\vbeta) - C_t^*(\hat \vbeta))
   = \frac1n \sum_{i=1}^n (G_i + 1) \Big[\int_0^t (\vbeta - \hat \vbeta)' \vX_i(s) \d N_i(s) - \int_0^t \log \frac{S_0(s, \vbeta)}{S_0(s, \hat \vbeta)} \d N_i(s)\Big]
 \end{align*}
 whose compensator is  $\tilde X^*(t, \vbeta) = \frac1n \sum_{i=1}^n [\int_0^t (\vbeta - \hat \vbeta)' \vX_i(s) \d N_i(s) - \int_0^t \log \frac{S_0(s, \vbeta)}{S_0(s, \hat \vbeta)} \d N_i(s)]$.
 
 Indeed, it turns out that functions of the form
 $$ m: t \mapsto \sum_{i=1}^n G_i \int_0^t k_{n,\vbeta,i}(s) \d N_i(s)  $$
 are martingales with respect to the filtration given by
 $$\mc F_t = \sigma\{N_i(s), Y_i(s), \vX_i(s), G_i \cdot N_i(v): 0 \leq s \leq \tau, \ 0 \leq v \leq t, \ i=1,\dots, n \},$$
 where the function $k_{n,\vbeta,i}$ is measurable with respect to $\mc F_0$.
 To verify this, we consider for $0 \leq r \leq t$
 \begin{align*}
  E[ m(t) \mid \mc F_r ] = \sum_{i=1}^n G_i \int_0^r k_{n,\vbeta,i}(s) \d N_i(s) + \sum_{i=1}^n \int_r^t k_{n,\vbeta,i}(s) \d E(G_i N_i(s) \mid \mc F_r) = m(r)
 \end{align*}
 due to  $E(G_i N_i(s) \mid \mc F_r) =   E(G_i N_i(s) \mid \mc F_0) =  E(G_i \mid \mc F_0) N_i(s) = 0$ for $r \leq s$.
 See \cite{dobler2016nonparametric} or \cite{bluhmki18b} for similar arguments in a nonparametric context.
 Similarly, it can be shown that $m$ has the predictable variation process given by
 $ \langle m \rangle(t) = \sum_{i=1}^n \int_0^t k_{n,\vbeta,i}^2 (s) \d N_i(s), $
 and the optional variation process given by
 $ [ m ](t) = \sum_{i=1}^n G_i^2 \int_0^t k_{n,\vbeta,i}^2 (s) \d N_i(s). $
 
 Thus, the predictable variation process of $X^*(t, \vbeta)$ is given by
 \begin{align}
  \langle X^*( \cdot, \vbeta) - \tilde X^*( \cdot, \vbeta) \rangle(t) 
   & = \frac1{n^2} \sum_{i=1}^n \int_0^t \Big[ (\vbeta - \hat \vbeta)' \vX_i(s) -  \log \frac{S_0(s, \vbeta)}{S_0(s, \hat \vbeta)} \Big]^2 \d N_i(s) \nonumber \\
   \label{eq:cp_VII21}
   & = \frac1{n^2} \sum_{i=1}^n \int_0^t \Big[ (\vbeta - \vbeta_0)' \vX_i(s) -  \log \frac{S_0(s, \vbeta)}{S_0(s, \vbeta_0)} \Big]^2 \d N_i(s)
   + O_p(n^{-2}),
 \end{align}
where the second equality follows from $\hat \vbeta - \vbeta_0 = O_p(n^{-1/2})$ in combination with the mean-value theorem applied to the function $\vbeta \mapsto S_0(s, \vbeta)$.
Its gradient, where the different partial derivatives are evaluated at different intermediate vectors $\tilde \vbeta$ \citep{feng13}, is bounded in probability.
Hence, we use Conditions~\ref{cond:abgk}(a)--(c) in combination with the conditional version of Lenglart's inequality (Section~II.5.2.1 in \citealt{abgk93}) and the fact that $n$ times the compensator of the  counting process integral in~\eqref{eq:cp_VII21} evaluated at $\tau$ converges in (conditional) probability to a finite function of $\vbeta$ to conclude that 
$n \langle X^*( \cdot, \vbeta) - \tilde X^*( \cdot, \vbeta) \rangle(\tau)$ too
converges to a finite function in $\vbeta$ in conditional probability as $n \rightarrow \infty$.
Furthermore, we use Lenglart's inequality again to show that the compensator $\tilde X^*( \cdot, \vbeta) $
converges in unconditional probability to
$$ f(\vbeta) = \int_0^\tau \Big[ (\vbeta - \vbeta_0)' s_1(s,\vbeta_0) - \log \frac{s_0(s,\vbeta)}{s_0(s,\vbeta_0)}  s_0(s, \vbeta_0) \Big] \lambda_0(s) \d s. $$
To see this, we again argue that 
$$\tilde X^*( \cdot, \vbeta) = \frac1n \sum_{i=1}^n \Big[\int_0^t (\vbeta - \vbeta_0)' \vX_i(s) \d N_i(s) - \int_0^t \log \frac{S_0(s, \vbeta)}{S_0(s, \vbeta_0)} \d N_i(s) \Big] + o_p(1) $$
for similar reasons as above.
Now, this counting process integral has an (unconditional) compensator that converges to 
$ f(\vbeta) $ in probability  as $n \rightarrow \infty$.
Another application of Lenglart's inequality can be used to show that, $n \rightarrow \infty$, $\tilde X^*( \cdot, \vbeta) \oPo f(\vbeta) $ in probability as well.
Finally, adding all arguments together, one last application of the conditional version of Lenglart's inequality implies that  $X^*( \cdot, \vbeta) \oPo f(\vbeta) $ in conditional probability  as $n \rightarrow \infty$.

Now, we can argue similarly to the proof of Theorem~VII.5.2.1 in \cite{abgk93}: 
By Condition~\ref{cond:abgk}(b)--(d), we have for any $\vbeta$
$$ \nabla f(\vbeta) = \int_0^\tau (\ve(s, \vbeta_0) -  \ve(s, \vbeta)) s_0(s, \vbeta_0) \lambda_0(s) \d s $$
and $\nabla f(\vbeta_0) = 0$. 
Furthermore, 
$$ - \nabla^2 f(\vbeta) = \int_0^\tau \vv(s, \vbeta_0) s_0(s, \vbeta_0) \lambda_0(s) \d s $$
which is positive semidefinite and positive definite for $\vbeta = \vbeta_0$; cf. Condition~\ref{cond:abgk}(e).

To make the following arguments less ambiguous, we use the subsequence principle for convergence in probability and fix, for any arbitrary subsequence $(n') \subset (n)$ another subsequence $(n'') \subset (n')$ such that the conditional convergence in probability $X^*( \cdot, \vbeta) \oPo f(\vbeta)$ given $\mc F_0$ holds almost surely along this subsequence $(n'')$. 
This convergence is point-wise in $\vbeta$ and the concave function $f(\vbeta)$ has a unique maximum at $\vbeta = \vbeta_0$.
The random function $X^*( \cdot, \vbeta)$ is also concave with a maximum at $\vbeta = \hat \vbeta^*$ if it exists.
We use Theorem~II.1 in Appendix~II of \cite{andersen1982} to conclude the uniformity of the convergence $X^*( \cdot, \vbeta) \oPo f(\vbeta) $.
For this reason, the maximizing value $\hat \vbeta^*$ of $X^*( \cdot, \vbeta)$ converges to the maximizing value $\vbeta_0$ of $f$ in conditional probability given $\mc F_0$ almost surely along the subsequence $(n'')$ chosen above.
We apply the subsequence principle another time to conclude that the conditional convergence in probability $\hat \vbeta^* \oPo \vbeta_0$ given $\mc F_0$ holds in probability as $n \rightarrow \infty$.
But the same convergence holds for $\hat \vbeta$, so the distance $\| \hat \vbeta^* - \hat \vbeta\| $ becomes arbitrarily small in conditional probability.

\medskip

\noindent \textbf{Proof of 2.}
Consider $\tilde D \vU_\tau^*(\tilde \vbeta) = (\nabla (\vU_\tau^{*})^{(1)} (\vbeta_1)', \dots, \nabla (\vU_\tau^{*})^{(p)} (\vbeta_p)')' $
where each $\vbeta_i$ is on the line segment between $\hat\vbeta^*$ and $\hat\vbeta$.
We again use martingale theory to prove the desired conditional convergences.
Without loss of generality, let us consider the complete matrix $\tilde D \vU_t^{*} (\tilde \vbeta)$ at only one intermediate vector $\tilde \vbeta$ on the line segment between $\hat \vbeta^*$ and $\hat \vbeta$
because if this matrix converges, then also each row converges, and hence the collection of several such rows converge as desired.

We make use of the following decomposition:
$$- \frac1n \tilde D \vU_\tau^*(\tilde \vbeta) = \frac1n \sum_{i=1}^n (G_i + 1) \int_0^t \vV(s, \tilde \vbeta) \d N_i(s) = \frac1n \sum_{i=1}^n G_i \int_0^t \vV(s, \tilde \vbeta) \d N_i(s) + \frac1n \int_0^t \vV(s, \tilde \vbeta) \d \sum_{i=1}^n N_i(s). $$
The first term on the right-hand side is asymptotically equivalent to 
$\frac1n \sum_{i=1}^n G_i \int_0^t \vV(s, \hat \vbeta) \d N_i(s) $;
indeed, $\| \tilde \vbeta - \hat \vbeta \| \oPo 0$ given $\mc F_0$ in conditional probability as $n \rightarrow \infty$ and $\vV$ converges uniformly to $\vv$ in both arguments. 
The remaining term $\frac1n \sum_{i=1}^n |G_i| N_i(t) $ is $O_p(1)$ and, therefore, does not matter.

Hence, we may as well focus on the square-integrable martingale
$\frac1n \sum_{i=1}^n G_i \int_0^t \vV(s, \hat \vbeta) \d N_i(s) $.
By using similar martingale arguments as above, i.e.\ Lenglart's inequality, it can again be shown that this term is asymptotically negligible.

It remains to analyse $ \frac1n \int_0^t \vV(s, \tilde \vbeta) \d \sum_{i=1}^n N_i(s). $
But, after having again argued why $\tilde \vbeta$ can be replaced with $\hat \vbeta$ and given $\mc F_0$, this term is deterministic and its unconditional asymptotic bevahiour is known:
it converges in probability to $\int_0^t \vv(s, \beta_0) s_0(s, \beta_0) \lambda_0(s) \d s$; 
cf. the proof of Theorem~VII.2.2 in \cite{abgk93}.
We conclude that $- \frac1n \tilde D \vU_\tau^*(\tilde \vbeta) \oPo \int_0^t \vv(s, \beta_0) s_0(s, \beta_0) \lambda_0(s) \d s$ in conditional probability given $\mc F_0$ as $n \rightarrow \infty$.

\medskip

\noindent \textbf{Proof of 1.}
We make use of the fact that $n^{-1/2} \vU_t^*(\hat \vbeta) = n^{-1/2} \sum_{i=1}^n G_i \int_0^\tau [ \vX_i(s) - \vE(s, \hat \vbeta) ] \d N_i(s)$ defines a square-integrable martingale with respect to the filtration $(\mc F_s)_s$.
This can be shown in the same way as for the other martingales above. 
Its predictable variation process is given by
$$n^{-1} \langle \vU_{(\cdot)}^*(\hat \vbeta) \rangle (t) =  \frac1n \sum_{i=1}^n  \int_0^t [ \vX_i(s) - \vE(s, \hat \vbeta) ]^{\otimes 2} \d N_i(s). $$
Similarly as before, this function is (unconditionally) asymptotically equivalent to
\begin{align*}
 & \frac1n \sum_{i=1}^n  \int_0^t [ \vX_i(s) - \vE(s, \vbeta_0) ]^{\otimes 2} \d N_i(s) \\
 & =  \frac1n \sum_{i=1}^n  \int_0^t [ \vX_i(s) - \vE(s, \vbeta_0) ]^{\otimes 2} \d M_i(s)
 +  \frac1n \sum_{i=1}^n  \int_0^t [ \vX_i(s) - \vE(s, \vbeta_0) ]^{\otimes 2} S_0(s, \vbeta_0) \lambda_0(s) \d s .
\end{align*}
The second term on the right-hand side converges in probability to $\int_0^t \vv(s, \beta_0) s_0(s, \vbeta_0) \lambda_0(s)$ while the remaining martingale term vanishes asymptotically:
its predictable variation process is given by 
$$ \frac1{n^2} \sum_{i=1}^n  \int_0^t [ \vX_i(s) - \vE(s, \vbeta_0) ]^{\otimes 4}  S_0(s, \vbeta_0) \lambda_0(s) \d s , $$
where $ {\bf{b}}^{\otimes 4} $ is basically  the array of all pairs of entries of the matrices $\bf{b}^{\otimes 2}$ and $\bf{b}^{\otimes 2}$.
Clearly, this predictable variation goes to zero in probability since the $\vX_i$ are bounded and the functions $\vE$ and $S_0$ converge uniformly in probability to bounded functions.
It remains to apply Rebolledo's martingale central limit theorem (Theorem~II.5.1 in \citealp{abgk93}) to conclude the asymptotic normality of 
$n^{-1/2} \vU_\tau^*(\hat \vbeta)$.
To this end, we again use the subsequence principle.
We see that, given $\mc F_0$ and along subsequences, the conditions of Rebolledo's theorem are satisfied almost surely; particularly the convergence of the predictable variation process of the square-integrable martingale
$t \mapsto n^{-1/2} \vU_t^*(\hat \vbeta)$.
Hence, almost surely along any subsequence, this process converges in distribution on the Skorokhod space $D[0,\tau]$ to a zero-mean Gaussian martingale whose covariance function is determined by the limit of the predictable variation process.
We have thus shown that $n^{-1/2} \vU_\tau^*(\hat \vbeta)$ converges in distribution to a random vector with a multivariate normal distribution.
Another application of the subsequence principle transfers the result to conditional convergence in probability along the original sequence $(n)$.
\end{proof}

\begin{proof}[Proof of Theorem~\ref{thm:wbs}]
 We only prove the assertion on the ``wild bootstrapping the score equations" approach. The applicability of the \cite{lin93} multiplier scheme follows along the same lines and is in fact more easily to prove because less applications of the mean-value theorem are required.
 All in all, we will make use of similar martingale arguments for the wild bootstrapped estimators as in the proof of Lemma~\ref{lemma:prelim}.
 To increase readability, some repeating arguments are omitted.
 
 %
 %
 As a first step, paralleling the proof of Theorem~VII.2.3 in \cite{abgk93},
 we deduce a useful asymptotic representation of the first part of $\sqrt{n} (\hat \Lambda_0^*(t, \hat \vbeta^*) - \hat \Lambda_0(t, \hat \vbeta))$, i.e., of
 \begin{align*}
  & \sqrt{n} \sum_{i=1}^n (G_i+1) \int_0^t [S_0^{-1}(u, \hat \vbeta^*) - S_0^{-1}(u,\hat \vbeta)] \d N_i(u) \\
  & = - \sqrt{n} (\hat \vbeta^* - \hat \vbeta) \sum_{i=1}^n (G_i+1) \int_0^t   \vE(u,  \tilde \vbeta^*)S_0^{-1}(u, \tilde \vbeta^*) \d N_i(u).
 \end{align*}
 This equality holds due to a Taylor expansion around $\hat \vbeta$.
 Here, $\tilde \vbeta^*$ is on the line segment between $\hat \vbeta^*$ and $\hat \vbeta$.

  Note  that we can replace the intermediate $\tilde \vbeta^*$ vectors by $\hat \vbeta$ because the resulting error 
  \begin{align}
  \label{disp:proof_taylor}
   \sum_{i=1}^n (G_i+1) \int_0^t  [  \vE(u, \tilde \vbeta^*)S_0^{-1}(u, \tilde \vbeta^*)- \vE(u, \hat \vbeta)S_0^{-1}(u, \hat \vbeta) ] \d N_i(u)
  \end{align}
  converges to zero in conditional probability as $n \rightarrow \infty$.
  Indeed, due to Lemma~\ref{lemma:prelim} in combination with Conditions~\ref{cond:abgk}(a)--(c), the above difference is bounded by $ \frac1n \sum_{i=1}^n | G_i + 1| o_p(1) = o_p(1) $. 
  Hence, 
  $$\sqrt{n} (\hat \Lambda_0^*(t, \hat \vbeta^*) - \hat \Lambda_0(t, \hat \vbeta)) =  - \sqrt{n} (\hat \vbeta^* - \hat \vbeta) \Big[ \sum_{i=1}^n (G_i+1) \int_0^t  \vE(u, \hat \vbeta)S_0^{-1}(u, \hat \vbeta) \d N_i(u) + o_p(1) \Big] + W^*(t), $$
  where $W^*(t)= \sqrt{n} \sum_{i=1}^n G_i \int_0^t S_0^{-1}(u, \hat \vbeta) \d N_i(u)$.
  Of the term in square brackets, the following sum vanishes:
  $$ \sum_{i=1}^n G_i \int_0^t  \vE(u, \hat \vbeta)S_0^{-1}(u, \hat \vbeta) \d N_i(u) = \frac{1}{\sqrt{n}} \Big(  \frac{1}{\sqrt{n}} \sum_{i=1}^n G_i \int_0^t  \vE(u, \hat \vbeta)\frac{n}{S_0(u, \hat \vbeta)} \d N_i(u) \Big)$$
  because the term in brackets is a square-integrable martingale with respect to $(\mc F_t)_t$, and thus asymptotically Gaussian.
  We conclude that 
   $$\sqrt{n} (\hat \Lambda_0^*(t, \hat \vbeta^*) - \hat \Lambda_0(t, \hat \vbeta)) =  - \sqrt{n} (\hat \vbeta^* - \hat \vbeta) \Big[ \int_0^t  \ve(u, \vbeta_0) \lambda_0(u) \d u + o_p(1) \Big] + W^*(t). $$
  
It remains to analyze $\sqrt{n} (\hat \vbeta^* - \hat \vbeta)$ and $W^*(t)$ jointly.
Thereof, $\sqrt{n} (\hat \vbeta^* - \hat \vbeta)$
 is essentially a linear transformation of $\vU^*_\tau(\hat \vbeta)$.
 Thus, its asymptotic multivariate normality follows from the first two assertions of Lemma~\ref{lemma:prelim} in combination with Slutzky's lemma.
 
For the joint convergence of $\sqrt{n} (\hat \vbeta^* - \hat \vbeta)$ and $W^*(t)$,
 we consider the process $t \mapsto (n^{-1/2} \vU_t^*(\hat \vbeta), W^*(t))'$ which, for similar reasons as in the proof of Lemma~\ref{lemma:prelim} defines a square-integrable martingale with respect to the filtration $(\mc F_t)_{t}$.
 We wish to apply Rebolledo's martingale central limit theorem (with non-trivial initial sigma field $\mc F_0$)
 in order to obtain the desired joint conditional central limit theorem which holds in probability.
 To this end, we analyze the predictable covariation process:
 \begin{align}
 \label{eq:pred_cov_beta_W}
  \langle n^{-1/2} \vU_{(\cdot)}^*(\hat \vbeta), & W^* \rangle(t) 
   = \sum_{i=1}^n \int_0^t [\vX_i(u) - \vE(u, \hat \vbeta) ] S_0^{-1}(u,\hat \vbeta) \d N_i(u).
 \end{align}
 Approximating $\vE(u, \hat \vbeta)$ on the right-hand side by $\vE(u, \vbeta_0)$ (then using~\eqref{eq:uniform_conv_sk_beta_hat})
 and $S_0^{-1}(u,\hat \vbeta)$ by $S_0^{-1}(u,\vbeta_0)$,
 a Taylor expansion around $\vbeta_0$ and the WLLN show that~\eqref{eq:pred_cov_beta_W}
 is asymptotically equivalent to
 \begin{align}
  & \sum_{i=1}^n \int_0^t [\vX_i(u) - \vE(u, \vbeta_0) ] S_0^{-1}(u,\vbeta_0) \d N_i(u) \nonumber \\
  & = \sum_{i=1}^n \int_0^t [\vX_i(u) - \vE(u, \vbeta_0) ] S_0^{-1}(u, \vbeta_0) \d M_i(u) \label{eq:pred_cov_beta_W-2} \\
  & \quad + \sum_{i=1}^n \int_0^t [\vX_i(u) - \vE(u, \vbeta_0) ] S_0^{-1}(u,\vbeta_0) Y_i(u) \exp(\vX_i'(u) \vbeta_0) \d \Lambda_0(u). \label{eq:pred_cov_beta_W-3}
 \end{align}
 By definition of $S_0$ and $S_1$, the second term~\eqref{eq:pred_cov_beta_W-3} on the right-hand side is zero.
 The remaining term~\eqref{eq:pred_cov_beta_W-2} is a martingale with predictable variation
 \begin{align*}
   \sum_{i=1}^n \int_0^t [\vX_i(u) - \vE(u, \vbeta_0) ]^{\otimes 2} S_0^{-2}(u, \vbeta_0)
  Y_i(u) \exp(\vX_i'(u) \vbeta_0) \d \Lambda_0(u) 
   = \int_0^t \vV(t,\vbeta_0) S_0^{-1}(u, \vbeta_0) \d \Lambda_0(u) \oPo 0.
 \end{align*}
 Thus, Lenglart's inequality implies that the martingale~\eqref{eq:pred_cov_beta_W-2} also goes to zero in probability as $n \rightarrow \infty$.
 Hence, $\sqrt{n} (\hat \vbeta^* - \hat \vbeta)$ and $W^*$ are asymptotically independent.
 
 Likewise, the other predictable variation processes (conditionally) converge as follows:
 \begin{align*}
  \langle n^{-1/2} \vU_{(\cdot)}^*(\hat \vbeta) \rangle(t) & \oPo \vSigma_t =  \int_0^t v(u,\vbeta_0) s_0(u, \vbeta_0) \d \Lambda_0(u), \\
  \langle W^* \rangle(t) & \oPo \omega^2(t) := \int_0^t \frac{\d \Lambda_0(u)}{s_0(u,\vbeta_0)}
 \end{align*}
 given $\mc F_0$ in probability by Conditions~\ref{cond:abgk}(a) and~(b). Finally, applying Rebolledo's Theorem, it follows that the optional covariation process $I_t^*(\hat \vbeta)$ of $\vU_t^*(\hat \vbeta)$ converges in conditional probability towards $\vSigma_t$.
\end{proof}

\bibliographystyle{plainnat}
\bibliography{literatur}

\end{document}